\documentclass[12pt]{amsart}

\usepackage[
text={440pt,575pt},
headheight=9pt,
centering
]{geometry}

\usepackage{microtype}
\usepackage{amsfonts,amsmath,amssymb,amsthm}
\usepackage[mathscr]{euscript}
\usepackage{paralist}

\usepackage{xparse}

\usepackage{tabularx}
\usepackage{multirow}
\usepackage{makecell}
\usepackage{array}
\usepackage{adjustbox}
\usepackage{ifthen}

\usepackage{caption}
\usepackage{subcaption}
\usepackage[dvipsnames]{xcolor}

\usepackage{lipsum}

\usepackage{tikz}
\usetikzlibrary{backgrounds}
\usetikzlibrary{arrows.meta}

\usepackage{url}
\usepackage{hyperref}

\usepackage{etoolbox}
\makeatletter
\patchcmd{\nocite}{\ifx\@onlypreamble\document}{\iftrue}{}{}
\makeatother

\makeatletter
\g@addto@macro\bfseries{\boldmath}
\makeatother

\makeatletter
\g@addto@macro\@floatboxreset{\centering}
\makeatother

\theoremstyle{definition}
\newtheorem{definition}{Definition}[section]

\newtheorem{theorem}[definition]{Theorem}
\newtheorem*{theorem*}{Theorem}
\newtheorem{proposition}[definition]{Proposition}
\newtheorem{lemma}[definition]{Lemma}
\newtheorem{corollary}[definition]{Corollary}

\newtheorem*{fact*}{Fact}
\newtheorem{remark}[definition]{Remark}
\newtheorem{example}[definition]{Example}

\newtheorem{challenge}[definition]{Challenge}

\newtheorem{conjecture}[definition]{Conjecture}

\usepackage{adjustbox}

\makeatletter
\newsavebox{\@brx}
\newcommand{\llangle}[1][]{\savebox{\@brx}{\(\m@th{#1\langle}\)}%
  \mathopen{\copy\@brx\kern-0.5\wd\@brx\usebox{\@brx}}}
\newcommand{\rrangle}[1][]{\savebox{\@brx}{\(\m@th{#1\rangle}\)}%
  \mathclose{\copy\@brx\kern-0.5\wd\@brx\usebox{\@brx}}}
\makeatother

\newcommand{\compl}[1]{\overline{#1}}

\newcommand{\floor}[1]{\left\lfloor #1 \right\rfloor}

\newcommand{\w}[1]{\mathtt{#1}}
\newcommand{\id}{\mathrm{id}}

\newcommand{\restr}{\mathrm{restr}}
\newcommand{\contr}{\mathrm{contr}}
\newcommand{\marg}{\mathrm{marg}}
\newcommand{\cond}{\mathrm{cond}}

\renewcommand{\subset}{\subseteq}
\renewcommand{\searrow}{\!\downarrow\!}
\renewcommand{\nearrow}{\!\uparrow\!}
\renewcommand{\rho}{\varrho}

\newcommand{\Mod}[1]{\ (\mathrm{mod}\ #1)}

\newcommand{\gaussoidof}[1]{\llangle #1 \rrangle}

\DeclareMathSymbol{*}{\mathord}{symbols}{"03}
\newcommand{\A}{\mathcal{A}}
\renewcommand{\AA}{\mathfrak{A}}

\newcommand{\C}{\mathcal{C}}

\newcommand{\F}{\mathcal{F}}

\newcommand{\G}{\mathcal{G}}
\newcommand{\GG}{\mathfrak{G}}

\newcommand{\J}{\mathcal{J}}

\renewcommand{\O}{\mathcal{O}}
\newcommand{\R}{\mathcal{R}}

\renewcommand{\S}{S}

\newcommand{\X}{\mathcal{X}}

\usepackage{xspace}
\newcommand{\Lnenicka}{Ln{\v e}ni{\v c}ka\xspace}
\newcommand{\Matus}{Mat{\'u}{\v{s}}\xspace}

\definecolor{darkblue}{RGB}{0,0,160}
\hypersetup{
colorlinks,%
citecolor=black,%
filecolor=black,%
linkcolor=darkblue,%
urlcolor=darkblue
}

\newcommand{\graphedgestyle}[1]{\tikzset{#1/.style={draw=black}}}
\NewDocumentCommand{\graph}{>{\SplitList{,}}m O{scale=0.4}}{%
\begin{tikzpicture}[
  E1/.style = {draw=none},
  E2/.style = {draw=none},
  E3/.style = {draw=none},
  #2
]

\tikzset{every node/.style={draw,shape=circle,fill=black,minimum size=4pt,inner sep=0}}
\tikzset{every edge/.style={line width=1pt}}
\ProcessList{#1}{\graphedgestyle}

\node (N1) at ( 0,  0) {};
\node (N2) at ( 1, -1.618) {};
\node (N3) at (-1, -1.618) {};

\path (N1) edge [E1] (N2);
\path (N2) edge [E2] (N3);
\path (N3) edge [E3] (N1);

\end{tikzpicture}
}

\begin{document}

\title{Construction Methods for Gaussoids}
\author{Tobias Boege and Thomas Kahle}
\date{December 11, 2019}

\keywords{gaussoid, conditional independence, normal distribution, cube, minor}
\subjclass[2010]{05B99, 05B35, 60E05}

% 60E05   	Distributions: general theory
% 05B35   	Matroids, geometric lattices [See also 52B40, 90C27]
% 05B99   	None of the above, but in this section

% \thanks{The authors are supported by the DFG (314838170, GRK 2297,
%  ``MathCoRe'').}

\begin{abstract}
\noindent
The number of $n$-gaussoids is shown to be a double exponential
function in~$n$.  The necessary bounds are achieved by studying
construction methods for gaussoids that rely on prescribing $3$-minors
and encoding the resulting combinatorial constraints in a suitable
transitive graph.  Various special classes of gaussoids arise from
restricting the allowed $3$-minors.
\end{abstract}

\maketitle

\section{Introduction}

Gaussoids are combinatorial structures that encode independence among
Gaussian random variables, similar to how matroids encode independence
in linear algebra.  They fall into the larger class of \emph{CI
  structures} which are arbitrary sets of conditional independence
statements.  The work of Fero \Matus is in particular concerned with
special CI structures such as graphoids, pseudographoids,
semigraphoids, separation graphoids, etc.  In his works \Matus
followed the idea that conditional independence can be abstracted away
from concrete random variables to yield a combinatorial theory.  This
should happen in the same manner as matroid theory abstracts away the
coefficients from linear algebra.  His work \cite{Matus97} on minors
of CI structures displays the inspiration from matroid theory very
clearly.

In 2007, \Lnenicka and \Matus defined gaussoids~\cite{LnenickaMatus07}
of dimension $n$ as sets of symbols $(ij|K)$, denoting conditional
independence statements, which satisfy the following Boolean formulas,
called the \emph{gaussoid axioms}:
\begin{alignat*}{4}
\tag{G1} \label{ax:G1} & (ij|L)  &&\wedge (ik|jL) &&\Rightarrow (ik|L)  &&\wedge (ij|kL), \\
\tag{G2} \label{ax:G2} & (ij|kL) &&\wedge (ik|jL) &&\Rightarrow (ij|L)  &&\wedge (ik|L),  \\
\tag{G3} \label{ax:G3} & (ij|L)  &&\wedge (ik|L)  &&\Rightarrow (ij|kL) &&\wedge (ik|jL), \\
\tag{G4} \label{ax:G4} & (ij|L)  &&\wedge (ij|kL) &&\Rightarrow (ik|L)  &&\vee   (jk|L),
\end{alignat*}
for all distinct $i, j, k \in [n]$ and
$L \subseteq [n] \setminus ijk$.  Here and in the following, we use
the efficient ``\Matus set notation'' where union is written as
concatenation and singletons are written without curly braces.  For
example, $ijk$ is shorthand for $\{i\}\cup\{j\}\cup\{k\}$.

A gaussoid is \emph{realizable} if its elements are exactly the
conditional independence statements that are valid for some
$n$-variate normal distribution.  Realizability was characterized for
$n=4$ in \cite{LnenickaMatus07} and a characterization for $n=5$ is
open.  There is no general forbidden minor characterization for
realizability of gaussoids~\cite{Simecek06,Sullivant}.  We therefore
think about gaussoids as \emph{synthetic conditional independence} in
the sense of Felix Klein~\cite[Chapter~V]{klein2016elementary}.  This
view is inspired by the parallels to matroid theory.  The algebra and
geometry of gaussoids was developed with this in mind
in~\cite{Geometry}.  Gaussoids are also the singleton-transitive
compositional graphoids according
to~\cite[Section~2.3]{sadeghi2017faithfulness}.

In the present paper we view gaussoids as structured subsets of
$2$-faces of an $n$-cube, which is possible because the symbols
$(ij|K)$ exactly index those faces.  This point of view on CI statements
was taken before in~\cite{ThreeCounterexamples}.  It readily simplifies
the definition of a gaussoid, but it has several additional advantages.
For example, it makes the formation of minors more effective, as this
now corresponds to restricting to faces of the cube.  To start,
consider the usual 3-dimensional cube.  A \emph{knee} in the cube
consists of two squares that share an edge.  A \emph{belt} consists of
all but two opposing squares of the cube.  We give a combinatorial
definition of a gaussoid that uses knees and belts.  It is equivalent
to the original definition in \cite{LnenickaMatus07} since both
definitions only specify how a gaussoid looks in a $3$-face of the
$n$-cube and there both definitions yield the exact same 11
possibilities appearing in Figure~\ref{fig:3-gaussoid-symmetry}.

\begin{definition}\label{d:withBelts}
  An $n$-gaussoid is a set $\G$ of $2$-faces of the $n$-cube such that
  for any $3$-face $C$ of the $n$-cube it holds:
  \begin{enumerate}
  \item If $\G$ contains a knee of $C$, then it also contains the belt
    that contains that knee.
  \item If $\G$ contains two opposing faces of $C$, then it also
    contains a belt that contains these two faces.
  \end{enumerate}
  The \emph{dimension} of~$\G$ is $n$, the dimension of the ambient
  cube.  $\GG_n$ is the set of $n$-dimensional gaussoids and
  $\GG := \bigcup_{n \ge 3} \GG_n$ the set of all gaussoids.
\end{definition}

\begin{figure}
\input{belt-axioms.fig}
\label{fig:belt-axioms}
\end{figure}

This definition is illustrated in Figure~\ref{fig:belt-axioms} where
the premises are purple and the conclusions are in shades of green.
The ``or'' condition of the gaussoid axiom \eqref{ax:G4} is reflected
in the second item of the definition as there are two belts through
two opposing faces of a 3-cube.  As with the gaussoid axioms,
Definition~\ref{d:withBelts} applies certain closure rules in every
$3$-face of the $n$-cube, but whereas $\S_3$ acts on the axes of the
cube in the gaussoid axioms, the group acting on the two pictures in
Figure~\ref{fig:belt-axioms} is the full symmetry group of the
$3$-cube, $B_3$. This bigger group conflates the first three axioms
into the first picture.

There is a small caveat in our definition of dimension of a gaussoid
as it equals that of the ambient cube.  Due to this choice, the empty
set, for example, can represent gaussoids of different dimension, just
as there is one empty matroid for every ground set. Usually, this will
cause no confusion but it underlines the importance of keeping in mind
that a gaussoid is always a \emph{subset} of squares in some ambient
cube.

The gaussoid axioms and also Definition~\ref{d:withBelts} only work
with $3$-cubes.  This locality can be expressed as in
Lemma~\ref{lemma:puzzling}: For any $k\ge 3$, being an $n$-gaussoid is
equivalent to all restrictions to $k$-faces being $k$-gaussoids.  The
aim of this work is to explore \emph{gaussoid puzzling}, the reversal
of this idea, that is, constructing $n$-gaussoids by prescribing their
$k$-gaussoids.  The implementation hinges on an understanding of how
exactly the $k$-faces of the $n$-cube intersect, because these
intersections are obstructions to the free specification of
$k$-gaussoids.  In Section~\ref{s:puzzling} we encode these
obstructions in a graph and then Brooks' theorem gives access to large
independent sets, where gaussoids can be freely placed.  This yields a
good estimate of the number of gaussoids in
Theorem~\ref{theorem:bounds}.

As gaussoids are defined by puzzling it is natural to ask what
additional structures emerge if one restricts the available pieces.
We do this in Section~\ref{s:specialGaussoids} where we explore classes
of ``special gaussoids'' that arise by restricting the puzzling of
$3$-gaussoids to $\S_3$-invariant subsets of the $11$ available pieces.
Consequently all of these classes can be axiomatized with $\S_{n}$-invariant
axioms that prescribe only the structure of $3$-cubes (like the gaussoid
axioms).  We find natural connections to the theory of graphical models
and other nice interpretations.

\subsection*{Acknowledgement}
The authors are supported by the Deutsche Forschungsgemeinschaft
(314838170, GRK 2297, ``MathCoRe'').

\section{The cube}
\label{sec:hypercube-minors}

Consider the face lattice $\F^{n}$ of the $n$-cube.  This lattice
contains $\emptyset$, the unique face of dimension $-\infty$.  To
specify a face of non-negative dimension $k$, one needs to specify the
$k$ dimensions in which the face extends, and then the location of the
face in the remaining $n-k$ dimensions.  We employ two natural ways to
work with faces.  The first is \emph{string notation}.  In this
notation a face $F$ is an element of $\{\w0,\w1,\w{*}\}^{n}$ where the
$\w*$s indicate dimensions in which the face extends and the remaining
binary string determines the location; a $\w1$ at position $p$ means
that the face is translated along the $p$-th axis inside the cube.
This string notation naturally extends the binary string notation for
the vertices of the $n$-cube: if $F\in\{\w0,\w1,\w{*}\}^{n}$, then its
vertices are
\[
  \left\{ a \in \{\w0,\w1\}^{n} : a_{i} = F_{i} \text{ whenever }
    F_{i} \neq \w*\right\}.
\]

The second choice is \emph{set notation}.  In this notation, a face
$F = (I_{F}|K_{F})$ of dimension $k = |I_F|$ is specified by two sets
$I_{F}\subset [n]$ and $K_{F} \subset [n]\setminus I_{F}$, where
$I_{F} = \{i\in[n] : F_{i} = \w*\}$ and
$K_{F} = \{i\in[n] : F_{i} = \w1\}$.  In this context the symbol
$\oplus$ denotes symmetric difference and $\compl{L}$ denotes the
complement of $L$ in $[n]$.

The set of $k$-faces of the $n$-cube is~$\F^n_k$.
As in~\cite{Geometry}, the squares of the $n$-cube are denoted by
$\A_n := \F^n_2$. Of special interest in this article are also the
$3$-faces $\C_n := \F^n_3$.  The constructions in
Section~\ref{s:puzzling} based on Lemma~\ref{lemma:puzzling}
frequently exploit the following

\begin{fact*}\label{fact:shared-face-dichotomy} For $3 \le k \le m$, a
  $k$-face shares at most $\binom{k-1}{2} 2^{k-3}$ squares with an
  $m$-face or is already included in it.  In particular for $k=3$, if
  a cube shares more than a single square with an $m$-face, then it is
  already contained in it.
\end{fact*}

Minors are important in matroid theory and gaussoid theory.  When a
simple matroid is represented as the geometric lattice of its flats, a
minor corresponds to an interval of the
lattice~\cite[Theorem~4.4.3]{Welsh10}, which is again a geometric
lattice. For gaussoid minors the lattice is replaced by the set of
squares in the hypercube and the lattice intervals are replaced by
hypercube faces.

Minors for arbitrary CI structures have been studied for example
in~\cite{Matus97}.  There, a \emph{minor} of a CI structure
$\A \subset \A_{n}$ is obtained by choosing two disjoint sets
$L, M \subseteq [n]$ and performing \emph{restriction} to $LM$
followed by \emph{contraction} by $\compl{L}$.  Writing
$\A_{L} = \{(ij|K) \in \A_{n} : ijK \subset L\}$ these are
\begin{align*}
  \contr_L \A &= \{ (ij|K) \in \A_L : (ij|K\compl{L}) \in \A \} \subseteq \A_L, \\
  \restr_L \A &= \A \cap \A_L \subseteq \A_L.
\end{align*}
In~\cite{Geometry}, minors were also defined specifically for gaussoids
using statistical terminology with an emphasis on the parallels to
matroid theory. A minor is every set of squares arising from a
gaussoid via any sequence of \emph{marginalization} and
\emph{conditioning}:
\begin{align*}
  \marg_L \A &= \{ (ij|K) \in \A : L \subseteq \compl{ijK} \}    \subseteq \A_{\compl{L}}, \\
  \cond_L \A &= \{ (ij|K) \in \A_{\compl{L}} : (ij|KL) \in \A \} \subseteq \A_{\compl{L}}.
\end{align*}
These operations are dual to the ones defined by \Matus:
$\cond_L = \contr_{\compl{L}}$ and $\marg_L = \restr_{\compl{L}}$.
Furthermore, either operation can be the identity, $\restr_{[n]} = \id$
and $\contr_{[n]} = \id$, and finally, the two sets $L$ and $M$ in
\Matus' definition of minor can be decoupled:
$\contr_L \restr_{LM} = \restr_L \contr_{\compl{M}}$.  Thus both
notions of minor coincide.

% As outlined above,
Our aim is to provide a geometric intuition for the
act of taking a gaussoid minor. A~face $(L|M)$ of the $n$-cube is
canonically isomorphic to the $L$-cube by deleting from the $[n]$-cube
$\{\w0, \w1\}^{[n]}$ all coordinates outside of~$L$.  This deletion is
a lattice isomorphism
$\pi_{(L|M)} : \F^{n} \cap (L|M) \leftrightarrow \F^L$, with the face
lattice $\F^{L}$ of an $|L|$-dimensional cube.  We~can interpret
taking the minor $\restr_L \cond_M$ as an operation in the hypercube.

\begin{proposition} \label{prop:minor-via-face}
Let $\A \subseteq \A_n$, then $\restr_L \cond_M \A = \pi_{(L|M)}
(\A \cap (L|M))$.
\end{proposition}

\begin{proof}
Take $(ij|K') \in \restr_L \cond_M \A$. Then $ij$ and $K'$ can be seen as
subsets of $[n]$ and they satisfy $ijK' \subseteq L$ and $(ij|K'M) \in \A$.
From this it is immediate that $ij \subseteq L$ and $K'M \subseteq LM$.
Furthermore, $\compl{ijK'M} = \compl{ijK'} \cap \compl{M} \subseteq
L\compl{LM}$, hence $(ij|K'M) \subseteq (L|M)$ and $(ij|K') \in
\pi_{(L|M)} (\A \cap (L|M))$.

In the other direction, suppose that $(ij|K') \in \pi_{(L|M)} (\A \cap (L|M))$
and let $(ij|K)$ be its preimage under $\pi_{(L|M)}$. Then $(ij|K) \in
\A \cap (L|M)$ and it follows $ij \subseteq L$, $K \subseteq LM$ and also
$M \subseteq K$. Thus $K$ decomposes into $K = K'M$ where naturally
$K' \cap M = \emptyset$. This proves that $(ij|K') \in \restr_L \cond_M \A$.
\end{proof}

Proposition~\ref{prop:minor-via-face} compactly encodes the
definitions of minor. The following definition introduces notation
reflecting this as well as an opposite \emph{embedding}, which mounts
a set of squares from the $L$-cube into an $|L|$-dimensional face of a
higher hypercube.

\begin{definition} \label{def:minor}
\begin{inparaenum}[(1)]
\item \label{def:minor-minor}
For a set $\A \subseteq \A_n$ and $(L|M) \in \F_k^n$, the
\emph{$(L|M)$-minor} of $\A$ is the set $\A \searrow (L|M) :=
\pi_{(L|M)} (\A \cap (L|M)) \subseteq \A_L$.
A \emph{$k$-minor} is an $(L|M)$-minor with~$|L| = k$.

\item \label{def:minor-embedding}
For a set $\A \subseteq \A_L$ and $(L|M) \in \F_k^n$, the
\emph{embedding} of $\A$ \emph{into $(L|M)$} is the preimage
$\A \nearrow (L|M) := \pi_{(L|M)}^{-1}(\A) \subseteq \A_n$.
\end{inparaenum}
\end{definition}

\section{Gaussoid puzzles}
\label{s:puzzling}

Several theorems in matroid theory concern the (impossibility of a)
characterization of classes of matroids in terms of forbidden minors.
For CI structures such as gaussoids the definitions read as follows.
\begin{definition}
\begin{enumerate}[(1)]
\item A class $\AA \subseteq \bigcup_n 2^{\A_n}$ of sets of squares
is \emph{minor-closed} if with $\A \in \AA$ all minors of $\A$ belong
to $\AA$.

\item A set of squares $\X$ is a \emph{forbidden minor}
for a minor-closed class $\AA$ if it is minimal with the property
that it does not belong to $\AA$, in the sense that all its proper
minors do belong to $\AA$.

\item The $k$-dimensional structures in a minor-closed class
$\AA$ are its \emph{compulsory $k$-minors}.
\end{enumerate}
\end{definition}

It is easy to see that gaussoids are minor-closed, i.e.\ any $k$-minor
of an $n$-gaussoid is always a $k$-gaussoid. But even more is true:
given any set of squares in the $n$-cube, if all of its $k$-minors,
for any $k \ge 3$, are $k$-gaussoids, then the whole is an
$n$-gaussoid.  This claim is proved in Lemma~\ref{lemma:puzzling}.
The present section uses this property to construct gaussoids by
prescribing their $k$-minors.  Section~\ref{s:specialGaussoids}
investigates subclasses of gaussoids which have the same anatomy.
We formalize this property in

\begin{definition} \label{def:puzzle-property} A class
  $\AA = \bigcup_{n \ge n_0} \AA_n$ of sets of squares stratified by
  dimension, i.e.\ $\AA_n \subseteq 2^{\A_n}$, has a \emph{puzzle
    property} if it is minor-closed and its $n$-th stratum is
  generated via embeddings from the strata below $n$, i.e.\ if for
  some $\A \subseteq \A_n$ all its $k$-minors, $k < n$, are in $\AA_k$,
  then already $\A \in \AA_n$.  The lowest stratum $\AA_{n_{0}}$ is the
  \emph{basis} of $\AA$ and the puzzle property is \emph{based} in
  dimension~$n_0$.
\end{definition}

\begin{lemma} \label{lemma:puzzling}
The set of gaussoids has a puzzle property based in dimension $3$,
whose basis are the eleven $3$-gaussoids.
\end{lemma}

\begin{proof} Let $\G \subseteq \A_n$ and $3 \le k \le n$. We show that
$\G$ is an $n$-gaussoid if and only if $\G \searrow D$ is a $k$-gaussoid
for every $D \in \F_k^n$.
First consider the case $k = 3$. The gaussoid axioms are quantified
over arbitrary cubes $(ijk|L)$ together with an order on the set $ijk$
and each axiom refers to squares inside the cube $(ijk|L)$
only. Confined to this cube, the axioms state precisely that this
$3$-minor is a $3$-gaussoid.
The case of $k > 3$ is reduced to the statement for $k = 3$. Indeed,
all $3$-minors of $\G$ are gaussoids if and only if all $3$-minors of
$k$-minors of $\G$ are gaussoids, because those two collections of
minors both arise from the same set $\C_n$ of cubes of the $n$-cube.
\end{proof}

Turning Definition~\ref{def:puzzle-property} upside down, the
construction of an $n$-gaussoid can be seen as a high-dimensional
jigsaw puzzle.  The puzzle pieces are lower-dimensional gaussoids
which are to be embedded into faces of the $n$-cube. The difficulty
comes from the fact that every square is shared by $\binom{n-2}{k-2}$
$k$-faces. The minors must be chosen so that all of them agree on
whether a shared square is an element of the $n$-gaussoid under
construction or not. The incidence structure of $k$-faces in the
$n$-cube is important. We~study it via the following graph.

\begin{definition}
Let $Q(n,k,p,q)$, for $n \ge k \ge p \ge q$, be the undirected simple
graph with vertex set $\F_k^n$ and an edge between $D, F \in \F_k^n$
if and only if there is a $p$-face $S$ such that $\dim(D \cap S) \ge q$
and $\dim(F \cap S) \ge q$.
\end{definition}

The idea behind this definition is that for suitable choices of $p$ and $q$,
the faces indexed by an \emph{independent set} in these graphs will be just
far enough away from each other in the $n$-cube to allow \emph{free puzzling}
of $k$-gaussoids without one minor choice creating constraints
for other minors.
% The main structural result about these graphs is

\begin{theorem} \label{theorem:Qnkpq}
The graph $Q(n,k,p,q)$ is transitive, hence regular. It is complete
if and only if $n + q \le p + k$. The degree of any vertex can be
calculated as follows:
\begin{equation*}
  \deg Q(n,k,p,q) = -1 + \sum_{m,j \; \text{\eqref{eq:degree-condition}}}
    \binom{k}{j}2^{k-j}\binom{n-k}{k-j}\binom{n-2k+j}{m},
\end{equation*}
where the sum extends over $0 \le m \le n-k$ and $0 \le j \le k$ which
satisfy the feasibility and connectivity conditions
\begin{equation*}
  \tag{$\dagger$} \label{eq:degree-condition}
  n-2k+j \ge m
  \quad \wedge \quad
  p \ge m + 2q - \min\{q, j\}.
\end{equation*}
\end{theorem}

\begin{proof}
  The symmetry group $B_{n}$ acts on the $n$-cube as automorphisms of
  the face lattice.  The group action is transitive on $k$-faces for
  any $k$ and respects meet and join.  Therefore $B_{n}$ is a subgroup
  of the automorphisms of and acts transitively on the graph~$Q(n,k,p,q)$.

  The characterization of completeness rests on
  Lemma~\ref{lemma:distance-adjacent}.  Using the gap function
  $\rho_q$ defined there, it is shown that $\rho_q(D,F) \le p$ is
  equivalent to the adjacency of $D$ and $F$ in $Q(n,k,p,q)$ and that
  if $F'$ is a face with smaller gap, then $F'$ is adjacent to~$D$.
  Since $Q(n,k,p,q)$ is regular, it is complete if and only if some
  vertex is adjacent to all others. For that to happen, the vertex
  must be adjacent to one which has the largest gap to it. As shown in
  the lemma, the maximum of $\rho_q$ is $n-k+q$ and hence completeness
  is equivalent to $n-k+q \le p$.

  The exact degree also follows from
  Lemma~\ref{lemma:distance-adjacent}.  Fix any vertex $D$ of
  $Q(n,k,p,q)$.  By regularity it suffices to count the adjacent
  vertices $F$ of~$D$.  We subdivide vertices $F$ according to two
  parameters: $m = |(K_D \oplus K_F) \setminus I_D I_F|$ is a
  disagreement between $D$ and $F$ and $j = |I_D \cap I_F|$ is the
  number of common dimensions of $D$ and~$F$.  A priori, $m$ ranges in
  $[n-k]$ and $j$ ranges in $[k]$, but not all combinations allow $F$
  to be a $k$-face adjacent to~$D$.  First, we determine the pairs
  $(m,j)$ for which an adjacent $k$-face exists and then count how
  many of them exist for fixed parameters. Let
  $(m,j) \in [n-k] \times [k]$.  For $j = |I_{D}\cap I_{F}|$ it must
  hold that $n \ge 2k - j$, since $D$ and $F$ are $k$-faces.  Assuming
  this, $F$ can be constructed if and only if the $k-j$ dimensions in
  $I_{F}\setminus I_{D}$ leave enough space to create the prescribed
  disagreement of size~$m$.  As an inequality this is
  $n-k \ge m + (k-j)$, or $n-2k+j \ge m$.  Together with $m \ge 0$,
  this inequality already entails the condition $n \ge 2k - j$ imposed
  by the choice of~$j$.  Thus it is sufficient to require
  $n-2k+j \ge m$, which is the first condition
  in~\eqref{eq:degree-condition}.  Given a $k$-face $F$ with
  parameters $m$ and $j$, the existence of an edge between $D$ and $F$
  in $Q(n,k,p,q)$ imposes the condition
  Lemma~\ref{lemma:distance-adjacent}~(\ref{distadj-p-edge}), which is
  the right half of~\eqref{eq:degree-condition}.

  As for the counting, let $D$ be a fixed $k$-face and let
  $(m,j) \in [n-k] \times [k]$ satisfy~\eqref{eq:degree-condition}.
  We count the $k$-faces $F$ with parameters $m$ and~$j$.  There are
  $\binom{k}{j}$ ways to place the $\w*$s for $I_{F} \cap I_{D}$.  On
  $I_{D} \setminus I_{F}$, there are $2^{k-j}$ independent choices
  from $\{\w0,\w1\}$.  The choices so far fix $F$ in~$I_D$.  There are
  now $\binom{n-k}{k-j}$ choices for the remaining $\w*$s in
  $I_{F}\setminus I_{D}$.  Then $I_F$ is fixed.  Now to finish $F$, we
  may only place $\w0$s and $\w1$s in $[n] \setminus I_DI_F$ where $D$
  has only $\w0$s and $\w1$s as well. Among the remaining $n-2k-j$
  positions, a set of size $m$ must be chosen, where $F$ is already
  determined by the condition that it differs from~$D$.  On the
  remaining $n-2k-j-m$ positions, $F$ is determined by not differing
  from~$D$.  The feasibility of all the choices enumerated so far is
  guaranteed by~\eqref{eq:degree-condition}.  The tally is
\begin{gather*}
  \sum_{m,j \; \text{\eqref{eq:degree-condition}}}
    \binom{k}{j}2^{k-j}\binom{n-k}{k-j}\binom{n-2k+j}{m}.
\end{gather*}
Since $D$ is not adjacent to itself, which is uniquely described by
the feasible parameters $j = k$ and $m = 0$, subtracting $1$ concludes
the proof.
\end{proof}

\begin{lemma} \label{lemma:distance-adjacent} Let $D$, $F$ be
  $k$-faces and $\rho_q(D, F) := m + 2q -\min\{q,j\}$, with
  $j = |I_D \cap I_F|$ and $m = |(K_D \oplus K_F) \setminus I_D
  I_F|$. The following hold:
\begin{enumerate}[(1)]
\item \label{distadj-p-edge} $\rho_q(D, F) \le p$ if and only if $D$
  and $F$ are adjacent in $Q(n,k,p,q)$,
\item \label{distadj-range}
the range of $\rho_q$ is $[q, n-k+q]$,
\item \label{distadj-isotone}
$\rho_q$ is strictly isotone with respect to $q$, i.e.\ $\rho_q < \rho_{q+1}$,
\item \label{distadj-heredit}
for $D, D', F \in \F_k^n$ with $\rho_q(D,D') \le \rho_q(D,F)$,
if $D$ and $F$ are adjacent in $Q(n,k,p,q)$, then so are $D$ and~$D'$.
\end{enumerate}
\end{lemma}
\begin{proof}
  Given two $k$-faces $D$ and $F$, the ground set $[n]$ splits into
  three sets:
  \begin{inparaenum}[(i)]
  \item $(K_D \oplus K_F) \setminus I_D I_F$ of cardinality $m$ where
    both have $\w{0}$ and $\w{1}$ symbols only but differ,
  \item $I_D \cap I_F$ of cardinality $j$ of shared $\w{*}$ symbols, and
  \item \label{distadj-other} everything else, i.e.\ positions where
    $\w{0}$ and $\w{1}$ patterns agree or where $\w{0}$ and $\w{1}$
    are in one face and $\w{*}$ in the other.
  \end{inparaenum}
  In order to connect two $k$-faces in $Q(n,k,p,q)$, there needs to be
  a $p$-face which intersects either of them in at least
  dimension~$q$.  Such a face has to cover the set of size $m$ with
  $\w{*}$s, as otherwise it will not intersect both faces.
  Conversely, once $m$ is covered, a $0$-dimensional intersection with
  both faces is ensured by placing $\w0$s and $\w1$s appropriately.
  To achieve a $q$-dimensional intersection, $q$ $\w{*}$s have to be
  placed on $I_D$ and $I_F$ each.  By using the $j$ shared $\w{*}$s,
  one needs at least $2q - \min\{q,j\}$ further $\w{*}$s to construct
  a connecting $p$-face.  Thus $\rho_q(D,F)$ is the minimum dimension
  $p$ necessary to connect $D$ and $F$ in $Q(n,k,p,q)$.  This proves
  claim~(\ref{distadj-p-edge}).

  It is clear that $\rho_q$ is minimal when $m$ is minimal and $j$ is
  maximal.  This can be achieved simultaneously by choosing $F = D$,
  in which case $\rho_q(D,D) = q$.  Now consider the opposing face
  $D^\circ = (I_D, [n] \setminus K_D I_D)$ of~$D$.  The gap is
  $\rho(D, D^\circ) = n - |I_D| + 2q - \min\{q, |I_D|\} = n - k + q$
  assuming $D$ is a vertex of $Q(n,k,p,q)$ where in particular
  $|I_D| = k \ge q$.  Increasing this value would require reducing $j$
  since $m$ is already maximal.  Un-sharing $\w{*}$s with $D$ consumes
  positions inside the block of $\w{0}$s and $\w{1}$s in $d$ of size
  $n-k$ which reduces $m$ by an equal amount.  Hence $n-k+q$ is
  maximal.  Furthermore, by varying $m$ but keeping $j = k$, all
  values in the range $[q, n-k+q]$ can be attained, proving
  claim~(\ref{distadj-range}).

  Claim~(\ref{distadj-isotone}) follows from a straightforward calculation:
  \begin{align*}
    \rho_{q+1}(D,F) - \rho_q(D,F)
    &= 2 - (\min\{q+1,j\} - \min\{q,j\}) \\
    &= \begin{cases}
      2, & j \le q, \\
      1, & j \ge q+1.
    \end{cases}
  \end{align*}

  In the situation of claim~(\ref{distadj-heredit}), since $D$ and $F$
  are adjacent in $Q(n,k,p,q)$, we have
  $\rho_q(D,D') \le \rho_q(D,F) \le p$
  by~(\ref{distadj-p-edge}). Applying~(\ref{distadj-p-edge}) again
  in reverse proves the claim.
\end{proof}

\pagebreak

\begin{corollary} \label{cor:Qn3p2-bounds} $ $
\begin{enumerate}[(1)]
\item $Q(n,3,2,2)$ is complete for $n \le 3$. Otherwise its degree is
  $6(n-3) \allowbreak \le 6(n-2)$.
\item $Q(n,3,3,2)$ is complete for $n \le 4$. Otherwise its degree is
  $12(n-3)(n-4)+ \allowbreak 7(n-3)\allowbreak \le 12(n-1)(n-2)$.
  \qed
\end{enumerate}
\end{corollary}

\begin{remark}
  For the theory of gaussoids the cases $k=3, p=2,3, q=2$ are
  relevant.  We consider it an interesting problem to study growth of
  the degree formula for other parameters.  Certainly the graph can be
  complete, where the degree is as large as $\binom{n}{k}2^{n-k}$.  To
  construct large independent sets, one wants smaller degrees.  It is
  proved below that a maximal independent set in $Q(n,3,3,2)$ has
  cardinality in $\Theta(n2^n)$ of which one inequality follows from
  the degree formula.
\end{remark}

\begin{proposition} \label{prop:independent-gaussoids} Let $\F$ be
  an independent set in $Q(n,k,3,2)$, then the following inequality
  holds: $|\GG_n| \ge |\GG_k|^{|\F|}$.
\end{proposition}

\begin{proof}
  Let $D, F \in \F$. Since $\F$ is independent, there is no $3$-cube
  sharing a square with $D$ and with~$F$.  Since $k \ge 3$, also $D$
  and $F$ share no square.  Thus an assignment of $k$-gaussoids
  $\alpha: \F \to \GG_k$ lifts to a well-defined set of squares
  $\G := \bigsqcup_{D \in \F} \alpha D \nearrow D \subseteq \A_n$.
  The map $\alpha \mapsto \G$ is injective.

  To see that $\G$ is a gaussoid, we examine its $3$-minors. Let
  $C \in \C_n$ be arbitrary. In case $C$ is fully contained in some
  $D \in \F$, then clearly
  $\G \searrow C = (\alpha D \nearrow D) \searrow C = \alpha D \searrow
  C \in \GG_3$ since $\alpha D \in \GG_k$.  Otherwise $C$ can share at
  most one square with any face in~$\F$.  If~it shares no square with
  any element of~$\F$, then $\G \searrow C$ is empty, hence a
  gaussoid. Otherwise it shares a square with some face in $\F$ and thus
  cannot share a square with any other element of $\F$ because $\F$ is an
  independent set in $Q(n,k,3,2)$. In this case, $\G \searrow C$ is a
  singleton or empty and hence a gaussoid.
\end{proof}

\begin{proposition} \label{prop:independent-non-gaussoids} Let $\F$ be
  an independent set in $Q(n,k,2,2)$ and $c$ the maximum size of a set
  of mutually range-disjoint injections of $\GG_k$ into
  $2^{\A_k} \setminus \GG_k$.  Then
  $\frac{2^{|\A_n|}}{|\GG_n|} \ge c^{|\F|}$.
\end{proposition}

\begin{proof}
  The proof is analogous to
  Proposition~\ref{prop:independent-gaussoids} but uses the
  independent set to perturb any gaussoid injectively into $c^{|\F|}$
  non-gaussoids. Again, since $q=2$ and $\F$ is independent, an
  assignment $\alpha: \F \to 2^{\A_k}$ lifts uniquely via $\nearrow$
  to a subset of $\A_n$. Let $\{f_i\}_{i \in [c]}$ be a set of
  range-disjoint injections as in the claim. Consider the maps
  $\alpha': \F \to [c]$.  To each $\G \in \GG_n$ associate
  $H_{\alpha'} := \bigsqcup_{D \in F} f_{\alpha' D}(\G \searrow D)
  \nearrow D \subseteq \A_n$.

  Because the ranges of the $f_i$ are disjoint, the map
  $(\G, \alpha') \mapsto H_{\alpha'}$ is injective. None of the sets
  $H_{\alpha'}$ is a gaussoid since any $D \in \F$ certifies
  $H_{\alpha'} \searrow D = f_{\alpha' D}(\G \searrow D) \not\in
  \GG_k$.
\end{proof}

\begin{remark} \label{rem:independent-set-technique} The proofs of
  Propositions~\ref{prop:independent-gaussoids}
  and~\ref{prop:independent-non-gaussoids} exploit two properties of
  the class of gaussoids:
\begin{inparaenum}[(1)]
\item \label{def:independent-set-technique-puzzling}
it has a puzzle property, and
\item \label{def:independent-set-technique-trivials} the empty set and
  all singletons are in its basis.
\end{inparaenum}
The same technique does not work for realizable gaussoids because they
lack property~(\ref{def:independent-set-technique-puzzling}) and not
for graphical gaussoids (see Section~\ref{s:specialGaussoids}) because
they lack property~(\ref{def:independent-set-technique-trivials}).
Indeed their numbers can be shown to be single exponential.  For
realizable gaussoids, this follows from Nelson's recent breakthrough:
If a gaussoid is realizable with a positive-definite $n\times n$
covariance matrix $\Sigma$, then the $n\times 2n$ matrix
$(I_{n}\, \Sigma)$ both defines a vector matroid and identifies
the gaussoid.  By \cite[Theorem~1.1]{Nelson16} there are only
exponentially many realizable matroids and thus realizable gaussoids.
Nelson's bound features a cubic polynomial in the exponent, while
there are certainly $2^{\binom{n}{2}}$ realizable gaussoids coming from
graphical models.
\end{remark}

To get explicit bounds we apply the propositions for $k=3$.  To find
suitable independent sets in $Q(n,3,3,2)$ and $Q(n,3,2,2)$ we use
Brooks' Theorem~\cite{LovaszBrooks} and the degree bounds from
Corollary~\ref{cor:Qn3p2-bounds}.  Since the graphs are connected,
have degree at least $3$ but are not complete, there exists a proper
$\deg Q(n,3,3,2)$-coloring of $Q(n,3,3,2)$, and we can pick a color
class as an independent set~$\F$.  Its size is at least that of an
average color class:
\begin{align*}
  \frac{|\F_3^n|}{\deg Q(n,3,3,2)}
  \ge \frac{n(n-1)(n-2)}{6 \cdot 12(n-1)(n-2)} 2^{n-3}
  = \frac{n}{6^2} 2^{n-4}
  = \frac{n}{9} 2^{n-6}.
\end{align*}
For $Q(n,3,2,2)$, we find analogously
\begin{align*}
  \frac{|\F_3^n|}{\deg Q(n,3,2,2)}
  \ge \frac{n(n-1)(n-2)}{6 \cdot 6(n-2)} 2^{n-3}
  = \frac{n(n-1)}{6^2} 2^{n-3}
  = \frac{n(n-1)}{9} 2^{n-5}.
\end{align*}
% The obtained quantities can be rounded to integers.  Not doing so
% still gives a lower bound on the size of an independent set.
Proposition~\ref{prop:independent-gaussoids} now shows, using
$|\GG_3| = 11$ and $\log_2 11 \ge 3$, that there are at least
$11^{\frac{n}{9} 2^{n-6}} \ge 2^{\frac{n}{3} 2^{n-6}}$ $n$-gaussoids.
Similarly, Proposition~\ref{prop:independent-non-gaussoids} with
$c = \floor{\frac{64-11}{11}} = 4$ gives an upper bound on the ratio
of $n$-gaussoids of
$4^{\frac{n(n-1)}{9} 2^{n-5}} = 2^{\frac{n(n-1)}{9} 2^{n-4}}$.  We
have proved
\begin{theorem} \label{theorem:bounds} For $n \ge 5$, the number of
  $n$-gaussoids is bounded by
\[
  2^{\frac{1}{3} n 2^{n-6}} \le |\GG_n| \le \frac{2^{|\A_n|}}{2^{\frac{4}{9} n(n-1) 2^{n-6}}}.\]
\end{theorem}

\begin{remark} \label{remark:nelson}
  A simple way to obtain a weaker double exponential lower bound for
  the number of gaussoids was suggested to us by Peter Nelson,
  following a matroid construction of Piff and Welsh~\cite{PiffWelsh71}.
  Let $\R_k$ be the set of all $r$-subsets $S$ of $[n]$ for some~$r<n$
  such that $\sum_{i\in S} i \equiv k \Mod{n}$, for $0 \le k < n$.
  We view $S\in \R_k$ as a $2$-face $(ij|S \setminus ij)$ of the
  $n$-cube, where $i,j$ are the minimal elements of~$S$.
  In this way any subset of $\R_k$ defines a set of $2$-faces which
  is vacuously a gaussoid: by construction, no gaussoid axiom has
  both of its premises in~$\R_k$.
  %The axioms \eqref{ax:G1} and \eqref{ax:G4} are satisfied because
  %their premises concern sets of different sizes. The premises of
  %\eqref{ax:G2} correspond to the same set $S\in \R$ and thus only
  %one of them can be in $\R_k$. The premises of \eqref{ax:G3} have
  %symmetric difference {j,k}. This is precluded by
  %$\sum_{i\in S} i \equiv k \Mod{n}$ for $S\in \R_k$, as $ijL$ and
  %$ikL$ both being in $\R_k$ would imply $j \equiv k \Mod{n}$
  %which means $j = k$.
  Certainly $k$ can be chosen so that $|\R_k| \ge \frac{1}{n}\binom{n}{r}$
  and with $r=\floor{n/2}$ this gives at least $2^{\frac{1}{n}\binom{n}{r}}
  \in 2^{\Theta(n^{-3/2}2^{n})}$ gaussoids.
\end{remark}

Substituting $|\A_n| = \binom{n}{2} 2^{n-2}$ in
Theorem~\ref{theorem:bounds} gives an interval for the absolute number
of $n$-gaussoids for $n \ge 5$. It shows
$\log |\GG_n| \in \Omega(n 2^n) \cap \O(n^2 2^n)$.

We conclude this section by showing that the $n2^n$ order lower bound
is the best that the independent set construction in $Q(n,3,3,2)$ can
do.  The \emph{independence number} $\alpha (G)$ of a graph $G$ is the
maximal size of an independent set in~$G$.  Similarly, the
\emph{clique~number} $\omega (G)$ is the maximal size of a clique
in~$G$.  Since $Q(n,3,3,2)$ is transitive, the following inequality
holds~\cite[Lemma~7.2.2]{GodsilRoyle}:
\[
  \alpha (Q(n,3,3,2)) \le \frac{|\F_3^n|}{\omega (Q(n,3,3,2))}.
\]
Since $|\F_3^n| \in \Theta(n^3 2^n)$, it suffices to find a clique of
size $\Omega(n^2)$ in every $Q(n,3,3,2)$. Take the set of cubes
$\J := \{(1ij|) : ij \in \binom{[n]\setminus 1}{2}\}$.  This set has
cardinality $\binom{n-1}{2} \in \Theta(n^2)$ and any two elements
$D = (1ij|)$, $F = (1kl|)$ in it are connected by an edge in
$Q(n,3,3,2)$, since
$\rho_2(D, F) = m + 2\cdot 2 - \min\{2, j\} = 4 - \min\{2, j\} \le 3$
with $m = 0$ and $j \ge 1$.

\section{Special gaussoids}
\label{s:specialGaussoids}

Because of their puzzle property, gaussoids are the largest class
of CI structures whose $k$-minors are $k$-gaussoids.  The base case
of this definition are the eleven $3$-gaussoids arising from $3 \times 3$
covariance matrices of Gaussian distributions.  The $3$-gaussoids split
into five symmetry classes modulo $\S_3$ which we denote by letters
\texttt{E}, \texttt{L}, \texttt{U}, \texttt{B}, and \texttt{F}.
They are depicted in Figure~\ref{fig:3-gaussoid-symmetry}.

\begin{figure}
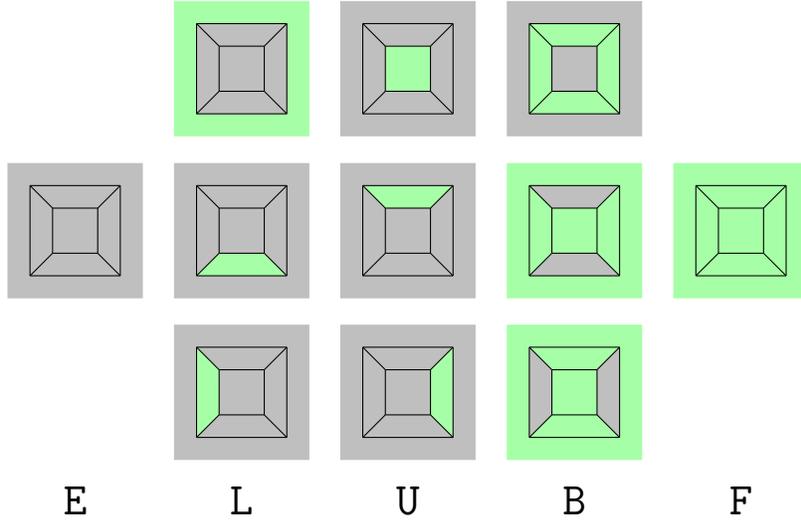

\begin{center}
\tikzstyle{face}=[draw=none]
\tikzstyle{inactive}=[face,gray,opacity=0.5]
\tikzstyle{active}=[face,green!70,opacity=0.5]
\noindent
\begingroup
% .5em column separation together with \\[.5em]
% inline row separation makes it look nice.
\setlength\tabcolsep{.5em}%
\begin{tabular}{ccccc}
&
  \begin{tikzpicture}[%
  % Vertex defaults: invisible
  A/.style     ={fill=none},
  B/.style     ={fill=none},
  C/.style     ={fill=none},
  D/.style     ={fill=none},
  E/.style     ={fill=none},
  F/.style     ={fill=none},
  G/.style     ={fill=none},
  H/.style     ={fill=none},
  % Edge defaults: normal
  AB/.style    ={draw=black},
  AD/.style    ={draw=black},
  AE/.style    ={draw=black},
  BC/.style    ={draw=black},
  BF/.style    ={draw=black},
  CD/.style    ={draw=black},
  CG/.style    ={draw=black},
  DH/.style    ={draw=black},
  EF/.style    ={draw=black},
  EH/.style    ={draw=black},
  FG/.style    ={draw=black},
  GH/.style    ={draw=black},
  % Face defaults: white
  front/.style ={white},
  bottom/.style={white},
  top/.style   ={white},
  back/.style  ={white},
  left/.style  ={white},
  right/.style ={white},
  % Supplied arguments override
  %
  scale=0.3,
  top/.style={inactive},
  back/.style={inactive},
  bottom/.style={inactive},
  right/.style={inactive},
  left/.style={inactive},
  front/.style={active},

  ]%
    \input{cube.tikz}%
  \end{tikzpicture}%
&
  \begin{tikzpicture}[%
  % Vertex defaults: invisible
  A/.style     ={fill=none},
  B/.style     ={fill=none},
  C/.style     ={fill=none},
  D/.style     ={fill=none},
  E/.style     ={fill=none},
  F/.style     ={fill=none},
  G/.style     ={fill=none},
  H/.style     ={fill=none},
  % Edge defaults: normal
  AB/.style    ={draw=black},
  AD/.style    ={draw=black},
  AE/.style    ={draw=black},
  BC/.style    ={draw=black},
  BF/.style    ={draw=black},
  CD/.style    ={draw=black},
  CG/.style    ={draw=black},
  DH/.style    ={draw=black},
  EF/.style    ={draw=black},
  EH/.style    ={draw=black},
  FG/.style    ={draw=black},
  GH/.style    ={draw=black},
  % Face defaults: white
  front/.style ={white},
  bottom/.style={white},
  top/.style   ={white},
  back/.style  ={white},
  left/.style  ={white},
  right/.style ={white},
  % Supplied arguments override
  %
  scale=0.3,
  top/.style={inactive},
  back/.style={active},
  bottom/.style={inactive},
  right/.style={inactive},
  left/.style={inactive},
  front/.style={inactive},

  ]%
    \input{cube.tikz}%
  \end{tikzpicture}%
&
  \begin{tikzpicture}[%
  % Vertex defaults: invisible
  A/.style     ={fill=none},
  B/.style     ={fill=none},
  C/.style     ={fill=none},
  D/.style     ={fill=none},
  E/.style     ={fill=none},
  F/.style     ={fill=none},
  G/.style     ={fill=none},
  H/.style     ={fill=none},
  % Edge defaults: normal
  AB/.style    ={draw=black},
  AD/.style    ={draw=black},
  AE/.style    ={draw=black},
  BC/.style    ={draw=black},
  BF/.style    ={draw=black},
  CD/.style    ={draw=black},
  CG/.style    ={draw=black},
  DH/.style    ={draw=black},
  EF/.style    ={draw=black},
  EH/.style    ={draw=black},
  FG/.style    ={draw=black},
  GH/.style    ={draw=black},
  % Face defaults: white
  front/.style ={white},
  bottom/.style={white},
  top/.style   ={white},
  back/.style  ={white},
  left/.style  ={white},
  right/.style ={white},
  % Supplied arguments override
  %
  scale=0.3,
  top/.style={active},
  back/.style={inactive},
  bottom/.style={active},
  right/.style={active},
  left/.style={active},
  front/.style={inactive},

  ]%
    \input{cube.tikz}%
  \end{tikzpicture}%
&

\\[.5em]

  \begin{tikzpicture}[%
  % Vertex defaults: invisible
  A/.style     ={fill=none},
  B/.style     ={fill=none},
  C/.style     ={fill=none},
  D/.style     ={fill=none},
  E/.style     ={fill=none},
  F/.style     ={fill=none},
  G/.style     ={fill=none},
  H/.style     ={fill=none},
  % Edge defaults: normal
  AB/.style    ={draw=black},
  AD/.style    ={draw=black},
  AE/.style    ={draw=black},
  BC/.style    ={draw=black},
  BF/.style    ={draw=black},
  CD/.style    ={draw=black},
  CG/.style    ={draw=black},
  DH/.style    ={draw=black},
  EF/.style    ={draw=black},
  EH/.style    ={draw=black},
  FG/.style    ={draw=black},
  GH/.style    ={draw=black},
  % Face defaults: white
  front/.style ={white},
  bottom/.style={white},
  top/.style   ={white},
  back/.style  ={white},
  left/.style  ={white},
  right/.style ={white},
  % Supplied arguments override
  %
  scale=0.3,
  top/.style={inactive},
  back/.style={inactive},
  bottom/.style={inactive},
  right/.style={inactive},
  left/.style={inactive},
  front/.style={inactive},

  ]%
    \input{cube.tikz}%
  \end{tikzpicture}%
&
  \begin{tikzpicture}[%
  % Vertex defaults: invisible
  A/.style     ={fill=none},
  B/.style     ={fill=none},
  C/.style     ={fill=none},
  D/.style     ={fill=none},
  E/.style     ={fill=none},
  F/.style     ={fill=none},
  G/.style     ={fill=none},
  H/.style     ={fill=none},
  % Edge defaults: normal
  AB/.style    ={draw=black},
  AD/.style    ={draw=black},
  AE/.style    ={draw=black},
  BC/.style    ={draw=black},
  BF/.style    ={draw=black},
  CD/.style    ={draw=black},
  CG/.style    ={draw=black},
  DH/.style    ={draw=black},
  EF/.style    ={draw=black},
  EH/.style    ={draw=black},
  FG/.style    ={draw=black},
  GH/.style    ={draw=black},
  % Face defaults: white
  front/.style ={white},
  bottom/.style={white},
  top/.style   ={white},
  back/.style  ={white},
  left/.style  ={white},
  right/.style ={white},
  % Supplied arguments override
  %
  scale=0.3,
  top/.style={inactive},
  back/.style={inactive},
  bottom/.style={active},
  right/.style={inactive},
  left/.style={inactive},
  front/.style={inactive},

  ]%
    \input{cube.tikz}%
  \end{tikzpicture}%
&
  \begin{tikzpicture}[%
  % Vertex defaults: invisible
  A/.style     ={fill=none},
  B/.style     ={fill=none},
  C/.style     ={fill=none},
  D/.style     ={fill=none},
  E/.style     ={fill=none},
  F/.style     ={fill=none},
  G/.style     ={fill=none},
  H/.style     ={fill=none},
  % Edge defaults: normal
  AB/.style    ={draw=black},
  AD/.style    ={draw=black},
  AE/.style    ={draw=black},
  BC/.style    ={draw=black},
  BF/.style    ={draw=black},
  CD/.style    ={draw=black},
  CG/.style    ={draw=black},
  DH/.style    ={draw=black},
  EF/.style    ={draw=black},
  EH/.style    ={draw=black},
  FG/.style    ={draw=black},
  GH/.style    ={draw=black},
  % Face defaults: white
  front/.style ={white},
  bottom/.style={white},
  top/.style   ={white},
  back/.style  ={white},
  left/.style  ={white},
  right/.style ={white},
  % Supplied arguments override
  %
  scale=0.3,
  top/.style={active},
  back/.style={inactive},
  bottom/.style={inactive},
  right/.style={inactive},
  left/.style={inactive},
  front/.style={inactive},

  ]%
    \input{cube.tikz}%
  \end{tikzpicture}%
&
  \begin{tikzpicture}[%
  % Vertex defaults: invisible
  A/.style     ={fill=none},
  B/.style     ={fill=none},
  C/.style     ={fill=none},
  D/.style     ={fill=none},
  E/.style     ={fill=none},
  F/.style     ={fill=none},
  G/.style     ={fill=none},
  H/.style     ={fill=none},
  % Edge defaults: normal
  AB/.style    ={draw=black},
  AD/.style    ={draw=black},
  AE/.style    ={draw=black},
  BC/.style    ={draw=black},
  BF/.style    ={draw=black},
  CD/.style    ={draw=black},
  CG/.style    ={draw=black},
  DH/.style    ={draw=black},
  EF/.style    ={draw=black},
  EH/.style    ={draw=black},
  FG/.style    ={draw=black},
  GH/.style    ={draw=black},
  % Face defaults: white
  front/.style ={white},
  bottom/.style={white},
  top/.style   ={white},
  back/.style  ={white},
  left/.style  ={white},
  right/.style ={white},
  % Supplied arguments override
  %
  scale=0.3,
  top/.style={inactive},
  back/.style={active},
  bottom/.style={inactive},
  right/.style={active},
  left/.style={active},
  front/.style={active},

  ]%
    \input{cube.tikz}%
  \end{tikzpicture}%
&
  \begin{tikzpicture}[%
  % Vertex defaults: invisible
  A/.style     ={fill=none},
  B/.style     ={fill=none},
  C/.style     ={fill=none},
  D/.style     ={fill=none},
  E/.style     ={fill=none},
  F/.style     ={fill=none},
  G/.style     ={fill=none},
  H/.style     ={fill=none},
  % Edge defaults: normal
  AB/.style    ={draw=black},
  AD/.style    ={draw=black},
  AE/.style    ={draw=black},
  BC/.style    ={draw=black},
  BF/.style    ={draw=black},
  CD/.style    ={draw=black},
  CG/.style    ={draw=black},
  DH/.style    ={draw=black},
  EF/.style    ={draw=black},
  EH/.style    ={draw=black},
  FG/.style    ={draw=black},
  GH/.style    ={draw=black},
  % Face defaults: white
  front/.style ={white},
  bottom/.style={white},
  top/.style   ={white},
  back/.style  ={white},
  left/.style  ={white},
  right/.style ={white},
  % Supplied arguments override
  %
  scale=0.3,
  top/.style={active},
  back/.style={active},
  bottom/.style={active},
  right/.style={active},
  left/.style={active},
  front/.style={active},

  ]%
    \input{cube.tikz}%
  \end{tikzpicture}%

\\[.5em]

&
  \begin{tikzpicture}[%
  % Vertex defaults: invisible
  A/.style     ={fill=none},
  B/.style     ={fill=none},
  C/.style     ={fill=none},
  D/.style     ={fill=none},
  E/.style     ={fill=none},
  F/.style     ={fill=none},
  G/.style     ={fill=none},
  H/.style     ={fill=none},
  % Edge defaults: normal
  AB/.style    ={draw=black},
  AD/.style    ={draw=black},
  AE/.style    ={draw=black},
  BC/.style    ={draw=black},
  BF/.style    ={draw=black},
  CD/.style    ={draw=black},
  CG/.style    ={draw=black},
  DH/.style    ={draw=black},
  EF/.style    ={draw=black},
  EH/.style    ={draw=black},
  FG/.style    ={draw=black},
  GH/.style    ={draw=black},
  % Face defaults: white
  front/.style ={white},
  bottom/.style={white},
  top/.style   ={white},
  back/.style  ={white},
  left/.style  ={white},
  right/.style ={white},
  % Supplied arguments override
  %
  scale=0.3,
  top/.style={inactive},
  back/.style={inactive},
  bottom/.style={inactive},
  right/.style={inactive},
  left/.style={active},
  front/.style={inactive},

  ]%
    \input{cube.tikz}%
  \end{tikzpicture}%
&
  \begin{tikzpicture}[%
  % Vertex defaults: invisible
  A/.style     ={fill=none},
  B/.style     ={fill=none},
  C/.style     ={fill=none},
  D/.style     ={fill=none},
  E/.style     ={fill=none},
  F/.style     ={fill=none},
  G/.style     ={fill=none},
  H/.style     ={fill=none},
  % Edge defaults: normal
  AB/.style    ={draw=black},
  AD/.style    ={draw=black},
  AE/.style    ={draw=black},
  BC/.style    ={draw=black},
  BF/.style    ={draw=black},
  CD/.style    ={draw=black},
  CG/.style    ={draw=black},
  DH/.style    ={draw=black},
  EF/.style    ={draw=black},
  EH/.style    ={draw=black},
  FG/.style    ={draw=black},
  GH/.style    ={draw=black},
  % Face defaults: white
  front/.style ={white},
  bottom/.style={white},
  top/.style   ={white},
  back/.style  ={white},
  left/.style  ={white},
  right/.style ={white},
  % Supplied arguments override
  %
  scale=0.3,
  top/.style={inactive},
  back/.style={inactive},
  bottom/.style={inactive},
  right/.style={active},
  left/.style={inactive},
  front/.style={inactive},

  ]%
    \input{cube.tikz}%
  \end{tikzpicture}%
&
  \begin{tikzpicture}[%
  % Vertex defaults: invisible
  A/.style     ={fill=none},
  B/.style     ={fill=none},
  C/.style     ={fill=none},
  D/.style     ={fill=none},
  E/.style     ={fill=none},
  F/.style     ={fill=none},
  G/.style     ={fill=none},
  H/.style     ={fill=none},
  % Edge defaults: normal
  AB/.style    ={draw=black},
  AD/.style    ={draw=black},
  AE/.style    ={draw=black},
  BC/.style    ={draw=black},
  BF/.style    ={draw=black},
  CD/.style    ={draw=black},
  CG/.style    ={draw=black},
  DH/.style    ={draw=black},
  EF/.style    ={draw=black},
  EH/.style    ={draw=black},
  FG/.style    ={draw=black},
  GH/.style    ={draw=black},
  % Face defaults: white
  front/.style ={white},
  bottom/.style={white},
  top/.style   ={white},
  back/.style  ={white},
  left/.style  ={white},
  right/.style ={white},
  % Supplied arguments override
  %
  scale=0.3,
  top/.style={active},
  back/.style={active},
  bottom/.style={active},
  right/.style={inactive},
  left/.style={inactive},
  front/.style={active},

  ]%
    \input{cube.tikz}%
  \end{tikzpicture}%
&

\\[.5em]
\Large\ttfamily E &
\Large\ttfamily L &
\Large\ttfamily U &
\Large\ttfamily B &
\Large\ttfamily F
\end{tabular}
\endgroup
\end{center}
\caption{
The eleven $3$-gaussoids in five symmetry classes mod $\S_3$
organized in columns. From left to right: the empty gaussoid
\texttt{E}, the lower singletons \texttt{L}, the upper
singletons \texttt{U}, the belts \texttt{B} and the full
gaussoid \texttt{F}.
}

\label{fig:3-gaussoid-symmetry}
\end{figure}

The special $\S_n$-invariant types of gaussoids in this section arise
from choosing subsets of these five symmetry classes to base a
puzzle property on.  Each of the 32 sets of bases can be
converted into axioms in the $3$-cube similar to the gaussoid axioms
\eqref{ax:G1}---\eqref{ax:G4}.  SAT solvers~\cite{sharpSAT,TodaSAT, miniSAT}
were used on the resulting Boolean formulas to enumerate or count
these classes and \cite{sage} was useful to create input files for
the~solver.  The listings can be found on our supplementary website
\href{https://gaussoids.de}{gaussoids.de}.
For nine classes an entry in the~\cite{OEIS} could be found.
Table~\ref{tab:special-gaussoids} is the main result of this section.
It summarizes the different types of gaussoids that arise from
the different~bases.

\begin{table}
\begin{tabularx}{\textwidth}{cl|l|c|l}
& Name & Count in dim. $3, 4, 5, \ldots$ & OEIS & Interpretation \\

\hline
\multicolumn{5}{l}{Fast-growing} \\
\hline

& \texttt{ELUBF} & 11, 679, 60\,212\,776 & --- & Gaussoids \\
& \texttt{ELUB}  & 10, 640, 59\,348\,930 & --- & --- \\
& \texttt{ELUF}  &  8, 522, 48\,633\,672 & --- & --- \\
& \texttt{ELU}   &  7, 513, 47\,867\,881 & --- & Required for Prop.~\ref{prop:independent-gaussoids} \\

\hline
\multicolumn{5}{l}{Incompatible} \\
\hline

& \texttt{LUB}      & 9, 111, 0, 0 & ---     & Vanishes for $n \ge 5$ \\
& \texttt{LUF}      & 7,  61, 1, 1 & ---     & Only \texttt{F} for $n \ge 5$ \\
& \texttt{LU}       & 6,  60, 0, 0 & ---     & Vanishes for $n \ge 5$ \\
& \texttt{\{L,U\}B} & 6,  15, 0, 0 & ---     & Vanishes for $n \ge 5$ \\
& \texttt{\{L,U\}F} & 4,   1, 1, 1 & ---     & Only \texttt{F} for $n \ge 4$ \\
& \texttt{EF}       & 2,   2, 2, 2 & \href{https://oeis.org/A007395}{A007395} & Only \texttt{E} or \texttt{F} for all $n$ \\

\hline
\multicolumn{5}{l}{Graphical} \\
\hline

& \texttt{E\{L,U\}BF} & 8, 64, 1\,024, 32\,768, 2\,097\,152 & \href{https://oeis.org/A006125}{A006125} & Undirected simple graphs \\
& \texttt{E\{L,U\}B}  & 7, 41,    388,  5\,789,    133\,501 & \href{https://oeis.org/A213434}{A213434} & Graphs without $3$-cycles \\
& \texttt{\{L,U\}BF}  & 7, 34,    206,  1\,486,     12\,412 & \href{https://oeis.org/A011800}{A011800} & Forests of paths on $[n]$ \\

%\hline
%\multicolumn{5}{l}{Partitional} \\
%\hline

& \texttt{E\{L,U\}F}, \texttt{EBF} & 5, 15, 52, 203, 877, 4\,140         & \href{https://oeis.org/A000110}{A000110} & Partitions of $[n]$ \\
& \texttt{E\{L,U\}},  \texttt{BF}  & 4, 10, 26,  76, 232,    764, 2\,620 & \href{https://oeis.org/A000085}{A000085} & Involutions on $[n]$ \\
& \texttt{EB}                      & 4,  8, 16,  32,  64,    128,    256 & \href{https://oeis.org/A000079}{A000079} & Subsets of $[n-1]$ \\

\hline
\multicolumn{5}{l}{Exceptional} \\
\hline

& \texttt{LUBF} & \makecell[tl]{10, 142, 1\,166, 12\,796, \\ 183\,772, 3\,221\,660} & --- & --- \\

\end{tabularx}
\caption{%
26 classes of special gaussoids categorized into four types. The remaining six
classes are described by one or zero letters of $\{\texttt{E}, \texttt{L},
\texttt{U}, \texttt{B}, \texttt{F}\}$ and belong to the Incompatible type,
as each of them is a subclass of a class found to be Incompatible.
}
\label{tab:special-gaussoids}
\end{table}

% \begin{lemma} \label{lemma:minor-algebra}
% Let $\A$ be a set of squares, $D$ a face and $\sigma \in \S_n$ a
% permutation. Then the following hold:
% \begin{inparaenum}[(1)]
% \item \label{lemma:Sn-minor}
% $\sigma \A \searrow \sigma D = \sigma(\A \searrow D)$, and
% \item \label{lemma:dual-minor}
% $\A^\circ \searrow D^\circ = (\A \searrow D)^\circ$.
% \end{inparaenum}
% \qed
% \end{lemma}

The classes \texttt{E}, \texttt{B} and \texttt{F} are themselves
closed under duality, while \texttt{L} and \texttt{U} are interchanged
by it.
% \comment{We invoked the above lemma before, but it's easy enough to
% just say 'it follows'.}
It follows that any one of the 32 classes is invariant under duality if
it contains either none of \texttt{L} and \texttt{U} or both of them.
On the remaining classes, duality acts by swapping \texttt{L} with
\texttt{U}. The combinatorial properties of the classes, e.g.~the
size, are unaffected by this action, hence \texttt{LB} and \texttt{UB}
are conflated to \texttt{\{L,U\}B} in
Table~\ref{tab:special-gaussoids}.

\subsection{Fast-growing gaussoids}

By Remark~\ref{rem:independent-set-technique}, the construction of
doubly exponentially many members of a class of gaussoids requires
that the class has a puzzle property and that its basis includes
\texttt{ELU}. This explains the rapid growth of all four classes of
this type.

\subsection{Incompatible minors}

As a consequence of Definition~\ref{def:puzzle-property}, if there is
no gaussoid of dimension $k$ in a class, there are no gaussoids of any
dimension $\ge k$ in the class. Similarly, if the class contains only
the empty or full gaussoid in dimension~$k$, the members of dimension
$\ge k$ are the empty or full gaussoid as well.  Hence computations in
small dimension suffice to explain these classes.
Despite their simplicity, each of them provides higher compatibility
axioms.  For example the annihilation of \texttt{LUB} in dimension $5$
implies that every $5$-minor of a gaussoid contains an empty or a full
$3$-minor.  Or: a graphical $4$-gaussoid with no belts is full or
contains an empty $3$-minor.

\subsection{Graphical gaussoids}

Each undirected simple graph $G = ([n], E)$ defines a CI structure
$\gaussoidof G := \{(ij|K) \in \A_n : \text{$K$ separates $i$ and
$j$}\}$, where two vertices $i$ and $j$ are \emph{separated} by a set
$K$ if every path between $i$ and $j$ intersects~$K$.  These are the
\emph{separation graphoids} of~\cite{Matus97}.  They fulfill a
localized version of the global Markov property.  According to
\cite[Remark~2]{LnenickaMatus07}, separation graphoids are exactly the
gaussoids satisfying the \emph{ascension axiom}:
\begin{equation*}
  \tag{A} \label{ax:A} (ij|L) \Rightarrow (ij|kL), \qquad \forall\, i,j,k\in [n], L\subset [n]\setminus ijk.
\end{equation*}
Therefore we refer to them as \emph{ascending gaussoids}.  The
operation $G \mapsto \gaussoidof G$ is a bijection whose inverse
recovers the graph via its edges
$E = \{ij : (ij|*) \not\in \gaussoidof G\}$, where $(ij|*)$
abbreviates $(ij|[n] \setminus ij)$.  Any gaussoid in this section is
of the form $\gaussoidof G$ for some undirected simple graph~$G$.

Since~\eqref{ax:A} uses only $2$-faces of a single $3$-face of the
$n$-cube, being an ascending gaussoid is a puzzle property based in
dimension~$3$. Its basis are the ascending $3$-gaussoids. This was
shown by \Matus~\cite[Proposition 2]{Matus97} and in our terminology
it can be restated as follows

\begin{lemma} \label{lemma:ascending-forbids-lower}
A gaussoid is ascending if and only if \texttt{L} is a forbidden minor.
\qed
\end{lemma}

This shows that \texttt{EUBF} are the ascending gaussoids. Their duals
are \texttt{ELBF} and it is easy to see that their axiomatization
replaces~\eqref{ax:A} by the \emph{descension axiom}
\begin{equation*}
  \tag{D} \label{ax:D} (ij|kL) \Rightarrow (ij|L), \qquad \forall\, i,j,k\in [n], L\subset [n]\setminus ijk.
\end{equation*}
\texttt{EUBF}-gaussoids arise from undirected graphs via vertex
separation, i.e.\ $(ij|K) \in \gaussoidof G$ if and only if $i$ and
$j$ are in different connected components of $G \setminus K$.  Their
duals contain $(ij|K)$ if and only if $i$ and $j$ are in different
connected components in the induced subgraph on $ijK$.  Therefore we
call elements of $\texttt{EUBF} \cup \texttt{ELBF}$ \emph{graphical}
gaussoids.
For our classification purposes it is sufficient to study the
``\texttt{U}pper'' half of dual pairs.
% of their subclasses.

Our technique to understand \texttt{EUBF} and its subclasses has
already been used in~\cite{Matus97}: since the presence of an edge
$ij$ in $G$ is encoded by the non-containment
$(ij|*) \not\in \gaussoidof G$, the compulsory minors of
$\gaussoidof G$ of the form $\gaussoidof G \searrow (ijk|*)$ prescribe
induced subgraphs on vertex triples~$ijk$.  In the opposite direction,
however, the induced $3$-subgraphs of a graph do not in general reveal
the types of all minors $\gaussoidof G \searrow (ijk|L)$ in its
corresponding gaussoid.

\begin{example}
\label{e:cycle}
Consider the cycle $\vcenter{\hbox{%
\begin{tikzpicture}[scale=0.23]
        \tikzset{every node/.style={draw,shape=circle,fill=black,minimum size=4pt,inner sep=0}}
        \tikzset{every edge/.style={draw=black,line width=1pt}}
        
        \node (N1) at (-1, -1) {};
        \node (N2) at (-1, +1) {};
        \node (N3) at (+1, +1) {};
        \node (N4) at (+1, -1) {};
        
        \path (N1) edge (N2);
        \path (N2) edge (N3);
        \path (N3) edge (N4);
        \path (N4) edge (N1);
\end{tikzpicture}%
}}$
corresponding to the gaussoid $\{(13|24), (24|13)\}$.
Its $3$-minors are exclusively \texttt{E} and \texttt{U}.
The \texttt{U} minors arise precisely in the $3$-cubes
\[
\{\w{1***}\}, \{\w{*1**}\}, \{\w{**1*}\}, \{\w{***1}\}.
\]
All other $3$-minors are \texttt{E}.  This means that the $4$-cycle is
contained in \texttt{EUBF}, \texttt{EUB}, and \texttt{EU}.  To match
with Table~\ref{tab:special-gaussoids}, check that the $4$-cycle has
no induced $3$-cycle, corresponds to the partition $13|24$ of $\{1,2,3,4\}$,
and the involution $(1\; 3)(2\; 4) \in \S_{4}$.

This graph shows that the class of a gaussoid cannot be determined
by looking only at the induced subgraphs of~$G$. All
$3$-minors observable from induced subgraphs are~\texttt{U},
but the smallest class to which this gaussoid belongs is~\texttt{EU}.
\end{example}

\begin{example}
\label{e:star}
Consider the star $\vcenter{\hbox{%
\begin{tikzpicture}[scale=0.3,rotate=60]
        \tikzset{every node/.style={draw,shape=circle,fill=black,minimum size=4pt,inner sep=0}}
        \tikzset{every edge/.style={draw=black,line width=1pt}}
        
        \node (Ni) at ( 0,    0) {};
        \node (Nj) at ( 1,    0) {};
        \node (Nk) at (-0.5,  0.866) {};
        \node (Nl) at (-0.5, -0.866) {};
        
        \path (Ni) edge (Nj);
        \path (Ni) edge (Nk);
        \path (Ni) edge (Nl);
\end{tikzpicture}%
}}$
with interior node $1$ and leaves $2,3,4$.
It corresponds to the gaussoid
\[
  \{(23|1), (23|14), (24|1), (24|13), (34|1), (34|12)\}.
\]
Because the right-hand side of every element of the gaussoid
contains $1$, this gaussoid has the minor \texttt{F} in $\w{1***}$,
\texttt{E} in the opposite face $\w{0***}$ and \texttt{U}
everywhere else.
\end{example}

We now establish relationships of subclasses of \texttt{EUBF} with
known combinatorial objects.  For some the graph
$G$ is more convenient, for others it is the complement graph
$G^c$ which is more natural.  Figure~\ref{fig:3-graphicals} shows the
complement graphs corresponding to \texttt{E}, \texttt{U}, \texttt{B}
and \texttt{F} and is useful to keep in mind for the proof of
Theorem~\ref{thm:graphicalClasses}.

\begin{figure}
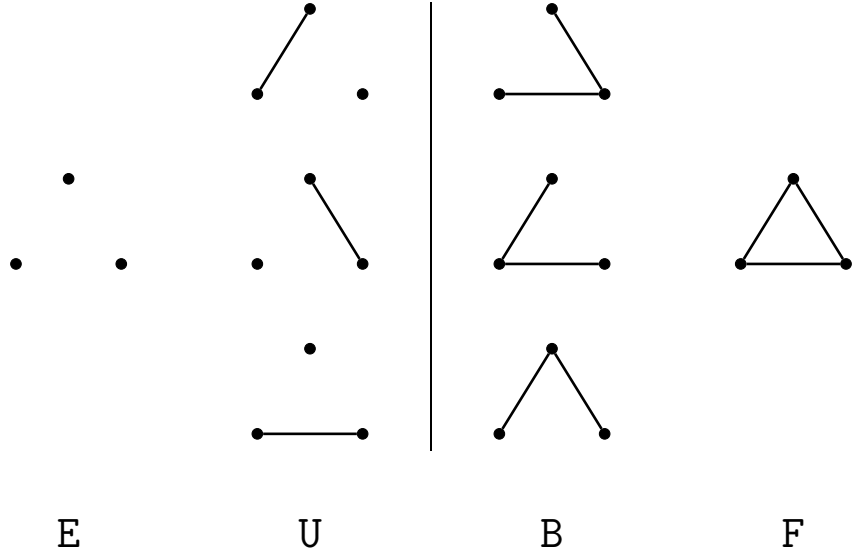

\begin{center}
\noindent
\begingroup
\setlength\tabcolsep{2em}%
\begin{tabular}{cc|cc}
&
\graph{E3}[scale=0.7]%
&
\graph{E1,E2}[scale=0.7]%
&

\\[2em]

\graph{}[scale=0.7]%
&
\graph{E1}[scale=0.7]%
&
\graph{E2,E3}[scale=0.7]%
&
\graph{E1,E2,E3}[scale=0.7]%

\\[2em]

&
\graph{E2}[scale=0.7]%
&
\graph{E1,E3}[scale=0.7]%
&

\\
% multicolumns prevent the column separator
\multicolumn{3}{c}{}
\\[1em]
\multicolumn{1}{c}{\Large\ttfamily E} &
\multicolumn{1}{c}{\Large\ttfamily U} &
\multicolumn{1}{c}{\Large\ttfamily B} &
\multicolumn{1}{c}{\Large\ttfamily F}
\end{tabular}
\endgroup
\end{center}
\caption{
The complementary graphs $G^c$ of $3$-gaussoids $\gaussoidof G$ organized
in symmetry classes mod $\S_3$ according to Figure~\ref{fig:3-gaussoid-symmetry}.
\texttt{E}, \texttt{U}, \texttt{B}, \texttt{F} index a partition of the
$\S_3$ orbits of all graphs on $3$ vertices. To obtain the diagram of
graphs $G$, flip the pictures over the vertical axis.
}

\label{fig:3-graphicals}
\end{figure}

\begin{theorem}
\label{thm:graphicalClasses}
The gaussoids in the class \texttt{EUBF} are in bijection with the
simple undirected graphs on $n$ vertices.  The subclasses distribute
as follows:
\begin{compactenum}[(1)]
  \item \texttt{EUB} contains exactly the gaussoids $\gaussoidof G$
  such that $G^{c}$ is $K_3$-free.
  \item \texttt{UBF} contains exactly the gaussoids $\gaussoidof G$
  such that each connected component of $G$ is a path.
  \item \texttt{EUF} contains exactly the gaussoids $\gaussoidof G$
  such that in $G^{c}$ each connected component is a clique,
  and hence corresponds to partitions of the vertex set $[n]$.
  \item \texttt{EU} is \texttt{EUF} where additionally every connected
  component of $G^{c}$ has at most two vertices.
\end{compactenum}
\end{theorem}

%\enlargethispage{2ex}
\begin{proof} The first statement summarizes the discussion in the
beginning of this section.
\begin{inparaenum}[(1)]
\item The graphs $G^{c}$ for $\gaussoidof G \in \texttt{EUB}$ are free
of triangles, as seen in Figure~\ref{fig:3-graphicals}.  If conversely
$G^{c}$ is triangle-free, then $\gaussoidof{G}$ does not have \texttt{F}
among its minors $(ijk|*)$.  By ascension, the cardinality of
$\gaussoidof G \searrow (ijk|L)$ is monotone in $L$ and thus no minor
of $\gaussoidof G$ is~\texttt{F}.

\item For $\gaussoidof G \in \texttt{UBF}$ we first show that every
vertex of $G$ has degree at most two.  Suppose a vertex $i$ was
adjacent to three distinct vertices $j, k, l$.  The subgraph induced
on $ijkl$ is the star discussed in Example~\ref{e:star} since $i$ has
degree three in this subgraph but none of its induced $3$-subgraphs
can be complete.  The corresponding gaussoid has \texttt{E} as a minor
and therefore this situation cannot arise in~$G$.  Therefore $G$ is a
disjoint union of cycles and paths.  If $G$ contains a cycle, let
$i,j,k$ be vertices of that cycle.  Since cycles are $2$-connected,
neither $(ij|k)$, nor $(ik|j)$, nor $(jk|i)$ is in $\gaussoidof G$.
Consequently, the minor $\gaussoidof G \searrow (ijk|) = \texttt{E}$
and thus $G$ contains no cycles.

Let now $G$ be a forest of paths.  Consider any three vertices
$i,j,k$.  If they are not all in the same connected component, say
$i,j$ are in different connected components, then
$(ij|), (ij|k) \in \gaussoidof G \searrow (ijk|)$ and thus this minor
is not \texttt{E}.  If $i,j,k$ are in the same connected component,
then, after suitable renaming, $i$ and $j$ on this path become
disconnected after removing~$k$.  Then $(ij|k) \in \gaussoidof G
\searrow (ijk|)$ and this minor is not \texttt{E}.  In both cases,
with ascension, it follows that for every $L \subset [n]\setminus ijk$
the minor $\gaussoidof G \searrow (ijk|L)$ is not \texttt{E}.

\item Let $\gaussoidof G \in \texttt{EUF}$.  The induced subgraphs of
$G^{c}$ on three vertices are precisely those which are closed
under the reachability relation within that subgraph. It is then clear
that every two vertices in the neighborhood of a fixed vertex are
connected by an edge, hence every connected component is a clique.

Let $G^{c}$ be a disjoint union of cliques and $i,j,k \in [n]$.
If they lie in pairwise different connected components, then the
$(ijk|*)$-minor of $\gaussoidof G$ is \texttt{E}; if exactly two of
them are in one component, then that minor is \texttt{U}. By ascension,
none of the minors $(ijk|L)$ can be \texttt{B} in these cases.
Finally suppose that $i,j,k$ are in the same connected component
and that $\gaussoidof G \searrow (ijk|L)$ is a belt containing,
say, $(ij|L)$ and $(ik|L)$ but not $(jk|L)$. Then $G$ contains a
path from $j$ to $k$ avoiding $L$. Because $jk$ is an edge in $G^c$,
this path contains another vertex $l \in [n]\setminus Lijk$ which is
adjacent to $j$. Since $jl$ is a non-edge in $G^c$ and $i$ and $j$
are in the same clique, $i$ and $l$ are adjacent in $G$.
This provides a path from $i$ over $l$ to $k$ in $G$ which avoids
$L$, contradicting the assumption.

\item Since $\texttt{EU} = \texttt{EUF} \cap \texttt{EUB}$, every
component of $G^c$, for $\gaussoidof G \in \texttt{EU}$, is a
clique but since there are also no induced $3$-cliques, the claim
follows.
\end{inparaenum}
\end{proof}

\begin{remark}
Motivated by the theory of databases, \Matus~\cite[Consequence~4]{Matus97}
also considered ascending gaussoids of chordal graphs. These have one
forbidden $4$-minor in addition to the compulsory $3$-minors \texttt{EUBF}.
In general, classes of graphs with prescribed induced subgraphs on
vertex sets $I$ can be studied from the gaussoid perspective by choosing
appropriate compulsory $(I|*)$-minors.
\end{remark}

The only graphical classes left are the subclasses of
$\texttt{EBF} = \texttt{EUBF} \cap \texttt{ELBF}$. These
\emph{bi-monotone} gaussoids are simultaneously ascending and descending because \texttt{L} and
\texttt{U} are forbidden.
% Such gaussoids are .
A bi-monotone gaussoid $\gaussoidof G$ is fixed by the symbols $(ij|)$
it contains. Such gaussoids can be seen as irreflexive, symmetric,
binary relations on $[n]$.

\begin{lemma} \label{lemma:EBF}
\texttt{EBF}-gaussoids are in bijection with the partitions of $[n]$.
\end{lemma}

\begin{proof}
Consider the gaussoid axioms under bi-monotonicity.  Axioms
\eqref{ax:G1}--\eqref{ax:G3} are trivial in the presence of ascension
and descension axioms, and \eqref{ax:G4} becomes
$(ij|) \Rightarrow (ik|) \vee (jk|)$. In terms of binary relations,
this is transitivity of the complement of~$\gaussoidof G$. Hence
\texttt{EBF}-gaussoids are complements of equivalence relations
on~$[n]$.
\end{proof}

A subclass of bi-monotone gaussoids is obtained by forbidding the
empty minor in addition to the forbidden singletons. The resulting
\texttt{BF}-gaussoids only have $3$-minors of cardinality at least
four and are called \emph{dense} gaussoids.

\begin{lemma} \label{lemma:BF}
The dense gaussoids \texttt{BF} correspond to involutions on $[n]$.
\end{lemma}

\begin{proof}
Let $\iota$ be an involution and $\gaussoidof G$ the
\texttt{EBF}-gaussoid associated, by Lemma~\ref{lemma:EBF}, to $\iota$'s
disjoint cycle decomposition.
Since $\iota$ is an involution, every cycle is either a fixed point or
a transposition. Take any two disjoint cycles $(i\; j)$ and $(k\; l)$
in $\iota$. Since $ij \cap kl = \emptyset$, no two symbols of the form
$(ij|K)$ and $(kl|M)$ appear in the same $3$-face, for any choice of
$K$ and $M$. This implies that every $3$-minor of $\gaussoidof G$ can
miss at most a single pair of opposite squares, which shows density.

Conversely, let $\gaussoidof G$ be a dense gaussoid. Consider the
partition corresponding to $\gaussoidof G$ as an
\texttt{EBF}-gaussoid. Assume there is a block containing at least
three distinct elements $i, j, k$, then $\gaussoidof G$ would not
contain $(ij|)$, $(ik|)$ and $(jk|)$, which is a contradiction to
$\gaussoidof G$ being dense at the $(ijk|)$-minor.
\end{proof}

\begin{lemma} \label{lemma:EB}%
An \texttt{EB}-gaussoid is defined by its characteristic vector with
respect to $(12|), (13|), (14|), \ldots, (1n|)$ and every such vector
defines an \texttt{EB}-gaussoid.
\end{lemma}

\begin{proof}
Let $\gaussoidof G$ be an \texttt{EB}-gaussoid and $i,j \neq 1$ be
distinct.  Consider the $(1ij|)$-minor of~$\gaussoidof G$.  Looking up
$(1i|)$ and $(1j|)$ in the characteristic vector, we can decide
whether $\gaussoidof G \searrow (1ij|)$ is empty or a belt. In either
case the containment of $(ij|)$ in $\gaussoidof G$ is determined by
the status of $(1i|)$ and~$(1j|)$.  Vice versa, this reconstruction
method freely defines a gaussoid all whose minors are necessarily
\texttt{E} or \texttt{B}.
\end{proof}

\begin{remark}
We consider it an interesting topic to determine properties beyond
combinatorics of the tamer graphical classes. For example, the
\texttt{EUBF}-gaussoids are precisely the positively orientable
gaussoids (see \cite[Section~5]{Geometry} for the precise definition),
their duals \texttt{ELBF} are the negatively orientable ones.  It can
also be shown that a \texttt{BF}-gaussoid $\gaussoidof G$ has exactly
$2^t$ orientations where $t$ is the number of transpositions in the
involution associated with~$\gaussoidof G$. All graphical gaussoids
are realizable.
\end{remark}

\subsection{The exceptional class}

The class \texttt{LUBF} remains mysterious.  We have tried various
arithmetic operations to transform the counts before searching OEIS,
but nothing emerged.  Unlike graphical gaussoids, there exist
non-orientable and hence non-realizable $\texttt{LUBF}$-gaussoids.
The following table lists their counts, for which likewise no
interpretation is known to the authors.

\begin{table}[ht]
\begin{tabular}{c||c|c|c|c|c|c}
$n$                 &  3 &   4 &      5 &       6 &        7 &           8 \\
\hline
all $\texttt{LUBF}$ & 10 & 142 & 1\,166 & 12\,796 & 183\,772 & 3\,221\,660 \\
non-orientable      &  0 &  42 &    210 &  1\,260 &  14\,700 &    355\,740
\end{tabular}
\end{table}

It is remarkable that all these numbers are divisible by $42$ and the
numbers for dimensions $\ge 5$ even by $210$.  As a first step towards
understanding whether some structure underlies these numbers, we pose

\begin{challenge}
Find or disprove the existence of a finite forbidden minor characterization
for non-orientable \texttt{LUBF}-gaussoids.
\end{challenge}

Our taxonomy of special gaussoids displays a trichotomy of growth
behaviors: double exponential, single exponential or bounded in the
dimension.  We do not know where the growth of \texttt{LUBF} falls.
Since the number of \texttt{LUBF}-gaussoids appears to grow slower
than the number of ascending gaussoids, we make the following

\begin{conjecture}
There is a single exponential upper bound for the number of
\texttt{LUBF}-gaussoids in fixed dimension $n$.
\end{conjecture}

Support for this conjecture comes from the fact that forbidding \texttt{E}
as a minor leads to a high density, that is many squares, in the resulting
gaussoids.  To see this take an independent set in $Q(n,3,2,2)$.
Each of the minors indexed by that set contains at least one $2$-face
and the independence ensures that no $2$-face is counted twice.
Thus an \texttt{LUBF}-gaussoid has at least
$\alpha (Q(n,3,2,2)) \ge \delta n^2 2^n$ elements, with a positive
constant~$\delta$ independent of $n$.  We suspect that containing a
positive fraction of all squares is sufficient for \texttt{LUBF} to
have single exponential size.

\leavevmode

\pagebreak

\nocite{OEIS}
\bibliographystyle{alpha}
\bibliography{kybernetika}

\bigskip \medskip

\noindent
\footnotesize {\bf Authors' addresses:}

\smallskip

\noindent Tobias Boege, OvGU Magdeburg, Germany,
{\tt tobias.boege@ovgu.de}

\noindent Thomas Kahle, OvGU Magdeburg, Germany,
{\tt thomas.kahle@ovgu.de}

\end{document}

%%% Local Variables:
%%% mode: latex
%%% TeX-master: t
%%% End: